\documentclass[a4paper,11pt]{article}
\usepackage{a4wide}
\usepackage[utf8]{inputenc}
\usepackage[T1]{fontenc}
\usepackage[english]{babel}
\usepackage{amsmath}
\usepackage{amssymb}
\usepackage{url}
\usepackage{multirow}
\usepackage{amsfonts}
\usepackage{amsthm}
\usepackage{indentfirst}
\usepackage{bbm}
\usepackage{color}
\usepackage{cite}

\mathchardef\re="023C
\mathchardef\im="023D

\newtheorem{theorem}{Theorem}

\newtheorem{defin}{Definition}

\newtheorem{propo}{Proposition}

\newtheorem{lema}{Lemma}

\newtheorem{thm}{Theorem}[section]
\newtheorem{prop}[thm]{Proposition}
\newtheorem{lem}[thm]{Lemma}
\newtheorem{defi}[thm]{Definition}
\newtheorem{cl}[thm]{Claim}
\newtheorem{cor}[thm]{Corollary}

\def\dif{\mathrm{d}}
\def\be{\begin{equation}}
\def\ee{\end{equation}}
\def\ra{\rightarrow}
\newcommand{\bcl}{\begin{cl}}
\newcommand{\ecl}{\end{cl}}
\newcommand{\bpf}{\begin{proof}}
\newcommand{\epf}{\end{proof}}
\newcommand{\bcor}{\begin{cor}}
\newcommand{\ecor}{\end{cor}}
\newcommand{\bthm}{\begin{thm}}
\newcommand{\ethm}{\end{thm}}

\newcommand{\B}{\mathbb{B}}

\newcommand{\Sp}{{\mathbb S}}

\newcommand{\supp}{\mathop{\mathrm{supp}\,}}

\newcommand{\IR}{\mathbb{R}}

\newcommand{\dd}{\,\mathrm{d}}

\numberwithin{equation}{section}

\begin{document}
\author{Miodrag Mateljevi\'c \thanks{Serbian Academy of Science and Arts, Kneza Mihaila 35, Belgrade, Serbia (\tt miodrag@matf.bg.ac.rs)}
\and
Nikola Mutavd\v{z}i\'c \thanks{Mathematical institute SANU, Kneza Mihaila 36, Belgrade, Serbia (\tt nikolam@matf.bg.ac.rs)}}
%[hyperbolic Poisson’s equation]
\title{On Lipschitz continuity and smoothness up to the boundary of solutions of hyperbolic
Poisson’s equation}

%\subjclass[2010]{Primary 30C62, 31C05 .}
%\keywords{the  Laplacian-gradient inequality, quasiconformal harmonic  mappings, Boundary behaviour of  partial derivatives, PDE of  the second order}

\maketitle
%\maketitle
%In communication with Arsenovic  and Khalfalah.\smallskip

\thanks{Research partially supported by MNTRS, Serbia,  Grant No. 174 032}

\begin{abstract}
%In this article, we investigate boundary behavior of hyperbolic harmonic functions, defined on unit ball and upper half-space in $\mathbb{R}^n$.
We solve the Dirichlet problem $\left.u\right|_{\mathbb{B}^n}=\varphi,$ for hyperbolic Poisson's equation $\Delta_h u=\mu$ where $\varphi\in L_1(\partial \mathbb{B}^n)$ and $\mu$ is a measure that satisfies a growth condition.

Next we present a short proof for Lipschitz continuity of solutions of certain hyperbolic Poisson's equations, previously established at \cite{ChenRas}.

In addition, we investigate some alternative assumptions on hyperbolic Laplacian, which are connected with Riesz's potential.  Also, local H\"{o}lder continuity is proved for solution of certain hyperbolic Poisson's equations.

We show that, if $u$ is  hyperbolic harmonic in the upper half-space, then $\frac{\partial u}{\partial y}(x_0,y)\to 0, y\to 0^+$, when boundary function $f$ of the functions $u$ is differentiable at the boundary point $x_0$. As a corollary, we show $C^1(\overline{\mathbb{H}^n})$ smoothness of a hyperbolic harmonic function, which is reproduced from the $C_c^1(\mathbb{R}^{n-1})$ boundary values.

\end{abstract}
\section{Introduction}

First, we give some notation. Let $x=(x_1,x_2,\ldots,x_n)\in\mathbb{R}^n$, and by $|x|$ we denote Euclidean norm of vector $x$, and $xy$ denotes the scalar product $\sum\limits_{j=1}^nx_jy_j$.\smallskip

For $R>0$, by $B(a,R)$ and $S(a,R)$ we denote the ball and the sphere in $\mathbb{R}^n$ with center at $a$ of radius $R$. By $B(R)$ and $S(R)$ we denote $B(0,R)$ and $S(0,R)$. We use $\mathbb{B}^n$ and $\mathbb{S}^{n-1}$ for $B(1)$ and $S(1)$.\smallskip

Let $\Omega$ be an open subset of $\mathbb{R}^n$.We use standard notation for function spaces on $\Omega$ (see \cite{gil.trud}).\\
$C^{k}(\Omega)$: the set of functions having all derivatives of order less then or equal to $k$ continuous in $\Omega$.\\
$C^k(\overline{\Omega})$: the set of functions in $C^k(\Omega)$ all of whose derivatives of order less than or equal to $k$ have continuous extensions to $\overline{\Omega}$.\\
$\supp u$: the support of $u$, the closure of the set on which $u\neq 0$.\\
$C^k_c(\Omega)$: the set of functions in $C^k(\Omega)$ with compact support in $\Omega$.\\
Let $x_0\in  D$, where $D$ is a bounded subset of $\mathbb{R}^n$ and $f$ is a function defined on $D$. For $0<\alpha<1$, we say that $f$ is {\it H\"{o}lder continuous} with exponent $\alpha$ at $x_0$ if $f$ is continuous at $x_0$ and\vspace{-2mm}
    $$
    \sup_{x\in D\setminus \{x_0\}}\frac{|f(x)-f(x_0)|}{|x-x_0|^{\alpha}}<+\infty.\vspace{-2mm}
    $$
    When $\alpha=1$, we say that $f$ is Lipschitz-continuous at $x_0$.\smallskip

Suppose that $D$ is not necessarily bounded. We say that $f$ is {\it uniformly H\"{o}lder continuous} with exponent $\alpha$ in $D$ if\vspace{-2mm}
    $$
    \sup_{\substack{x,y\in D,\\ x\neq y}}\frac{|f(x)-f(y)|}{|x-y|^{\alpha}}<+\infty,\quad 0<\alpha<1.\vspace{-2mm}
    $$
Let $\Omega$ be an open
set in $\mathbb{R}^n$ and $k$ a non-negative integer. The {\it H\"{o}lder spaces} $C^{k,\alpha}(\Omega) (C^{k,\alpha}(\overline{\Omega}))$ are
defined as the subspaces of $C^k(\Omega) (C^k(\overline{\Omega}))$ consisting of functions whose $k$-th order
partial derivatives are uniformly H\"{o}lder continuous (H\"{o}lder continuous) with exponent $\alpha$ in $\Omega$. $Lip(\Omega)$ denotes class of function which are Lipshitz continuous on the set $\Omega$.

\subsection{M\"{o}bius transformation in space}

\begin{defi}
  Inversion with respect to unit sphere in $\mathbb{R}^n$ is defined as\vspace{-1mm}
  $$J(a)=a^*=\frac{a}{|a|^*} \mbox{ for } a\neq 0 \mbox{ and } J(0)=\infty, J(\infty)=0.$$
\end{defi}

Following \cite{ahlMob}, the reflection  (inversion)  with respect to sphere $S(a,R)$ is given by\vspace{-1mm}
\begin{equation}
\sigma_a x=  a^*  +  R(a)^2  (x- a^*)^*\,.\vspace{-1mm}
\end{equation}

The following matrix represents orthogonal projection in $\mathbb{R}^n$ onto the one-dimensional subspace spanned by $x\in\mathbb{R}^n, x\neq 0$.\vspace{-1mm}
$$Q(x)_{ij}=\frac{x_i x_j}{|x|^2}.$$

Let $r_a:y\mapsto y'$ be the reflection with respect to the hyper plane orthogonal to $a\neq 0$ which contains the origin. It is easy to check that  $r_a= \mathrm{Id}- 2Q(a)$. Indeed,\vspace{-2mm} $$y'=(\mathrm{Id}- 2Q(a))y= y - 2 \frac{(ya) a}{|a|^2}.\vspace{-2mm}$$

Define
$T_a= r_a\circ \sigma_a= (\mathrm{Id}- 2Q(a))\sigma_a $. Check\vspace{-2mm}
\begin{equation*}
T_a x=\frac{(1-|a|^2)(x-a) - |x-a|^2 a}{[x,a]^2},\vspace{-1mm}
\end{equation*}
where   $[x,a]= |x| |x^*- a|= |a| |x-a^*|$   and    $[x,a]^2=1+|x|^2|a|^2 - 2x a$.
%\eeg
It is easy to check  that  $T_a a=0$, $T_a0=-a$   and $T_a$ maps  $[-\hat{a}, \hat{a}]$ onto itself. Also,\vspace{-1mm}
\begin{equation}\label{IdentMob1}
|T'_y x|= \dfrac{1-|y|^2}{[x,y]^2},\quad \mbox{and in particular}\quad |T'_y0|= \dfrac{1-|y|^2}{[0,y]^2}= 1-|y|^2.\vspace{-2mm}
\end{equation}
Denote by $\widehat{M}(\mathbb{B}^n)$ the set of all M\"{o}bius transformations in $\mathbb{B}^n$. For more informations about the M\"{o}bius transformations in $\mathbb{B}^n$,  see \cite{ahlMob}.
It is well known that\vspace{-2mm}
\begin{equation}\label{IdentMob2}
1-|T_ax|^2= \frac{(1-|a|^2)(1-|x|^2)}{[x,a]^2}, \vspace{-2mm}
\end{equation}
and  if  $\gamma \in\widehat{M}(\mathbb{B}^n)$, where $\widehat{M}(\mathbb{B}^n)$ is , then\vspace{-1mm}
$$\dfrac{|\gamma x-\gamma y|}{|[\gamma x,\gamma y]|} {=} \dfrac{|x-y|}{|[x,y]|}.\vspace{-1mm}$$
In the literature authors often use $-T_a$ instead of $T_a$. To avoid possible confusion we denote $-T_a$  with  $\varphi_a$.\smallskip
\subsection{Hyperbolic Poisson's equation}
Recall that hyperbolic Laplace operator in the $n$-dimensional hyperbolic ball
$\mathbb{B}^n$ for $n\geqslant 2$ is defined as\vspace{-3mm}
$$ \Delta_{h}u(x)= (1-|x|^2)^2\Delta u(x)+2(n-2)(1-|x|^2)\sum_{i=1}^{n}
x_{i} \frac{\partial u}{\partial x_{i}}(x).\vspace{-2mm} $$

In terms of the mapping $\varphi_a$, the {\it hyperbolic metric $d_h$} in $\mathbb{B}^n$ is given by\vspace{-2mm}
$$d_h(a,b)=\log\left(\frac{1+|\varphi_a(b)|}{1-|\varphi_a(b)|}\right)\vspace{-2mm}$$
for all $a,b\in\mathbb{B}^n$.

For all $\varphi\in\widehat{M}(\mathbb{B}^n)$, in \cite{ahlMob} it is proved that:\vspace{-2mm}
\begin{equation}\label{hLapInv}
  \Delta_h(u\circ\varphi)=\Delta_h u\circ\varphi.\vspace{-2mm}
\end{equation}

We say that $u:\mathbb{B}^n\to\mathbb{R}$ satisfies {\it hyperbolic Laplace equation} if $\Delta_h u =0$ in $\mathbb{B}^n$. Non-homogenous hyperbolic Laplace equation, i.e. $\Delta_h u =\psi, \psi\not\equiv 0$ is called {\it hyperbolic Poisson's equation}.
\subsection{Hyperbolic Poisson's kernel and Hyperbolic Green function}
Set $0<r<1$ and define \vspace{-2mm}
\be
g(r):=\int_r^1 \frac{(1-t^2)^{n-2}}{t^{n-1}}\dd t.\vspace{-2mm}
\ee

%Note that  for  $a\in \mathbb{S}$ holds $[x,a]^2+=1+ |x|^2- 2 xa= |x-a|^2$.\smallskip

Then hyperbolic Green function is given by\vspace{-2mm}
\be
g(x,y)=G_h(x,y)=g(|T_y x|)=g\left(\frac{|x-y|}{[x,y]}\right).
\ee

For  $n=2$ holds $g(r)=\log \frac{1}{r}$  and  for $n>2$ holds  $g(r)\sim \frac{1}{n-2}r^{2-n}$ if  $r\rightarrow 0$, and  $g(r)=O((1-r)^{n-1})$ if $r\rightarrow 1$.\smallskip

Let  $\dd\sigma$ be  the $(n-1)$-dimensional Lebesgue measure normalized so that $\sigma(\mathbb{S}^{n-1}) = 1$.

The Poisson-Szego kernel $P_h$ for hyperbolic laplacian  $\Delta_h$ is given by\vspace{-1mm}
\be
P_h (x, t) = \left(\frac{1 - |x|^2|}{|t - x|^2}\right)^{n-1},\vspace{-1mm}
\ee
which satisfies \vspace{-1mm}
$$\int_{\mathbb{S}^{n-1}}P_h (x, t) \dd\sigma (t) = 1.\vspace{-1mm}$$

Let us define hyperbolic Poisson's integral\vspace{-2mm}
\be
P_h [f](x) =
\int_{\mathbb{S}^{n-1}}P_h (x, t) f(t) \dd\sigma (t).\vspace{-2mm}
\ee
for $f\in L_1(\mathbb{S}^{n-1})$ and, also, the hyperbolic Green integral\vspace{-2mm}
$$G_{h}[\psi](x)=  \int_{\mathbb{B}^n} G_{h}(x,y) \psi(y) \dd\tau (y), \vspace{-2mm}$$
for appropriate functions $\psi$ and \vspace{-2mm}
$$\dd\tau(x)=\frac{\dd\nu(x)}{(1-|x|^2)^n}\vspace{-2mm}$$
where $\nu$ is the $n$-dimensional Lebesgue volume measure normalized so that $\nu(\mathbb{B}^n)=1$.
\subsection{Main results}

Dirchet problem is well understood   for smooth metrics.  For example, one can see Chapter IX of \cite{shen-yau}.
It turns out that this problem for hyperbolic  metric on the unit  ball with boundary data on the unit sphere is very interesting and it  is considered recently in \cite{ChenRas}. Here the metric density  goes to  $\infty $ near the boundary.
Among the other things,  J. Chen, M. Huang, A. Rasila and X. Wang used nice properties of M\"{o}bius transformation and   hyperbolic  Green function of  the unit  ball (described for example  in \cite{ahlMob,stoll}) and integral estimate.
%In \cite{ChenRas}, the authors investigated case of Dirichet problem for Laplace and Poisson's equation in the case of hyperbolic metric on the unit ball $\mathbb{B}^n$, which is singular on the boundary.
Precisely, the authors show that if $n\geqslant 3$ and $u\in
C^{2}(\mathbb{B}^{n},\IR^n) \cap C(\overline{\mathbb{B}^{n}},\IR^n )$ is a solution to the hyperbolic Poisson equation, then it has a
representation\vspace{-2mm}
\begin{equation}\label{RepHyp}u=P_{h}[\phi]-G_{h}[\psi],\vspace{-2mm}\end{equation}
provided that\vspace{-2mm}
$$u\mid_{\mathbb{S}^{n-1}}=\phi\quad \mbox{and}\quad
\int_{\mathbb{B}^{n}}(1-|x|^{2})^{n-1} |\psi(x)|\,d\tau(x)<\infty.\vspace{-1mm}$$
Here $P_{h}$ and $G_{h}$ denote Poisson and Green integrals with respect to $\Delta_{h}$, respectively. Furthermore, they  prove that functions of the form $u=P_{h}[\phi]-G_{h}[\psi]$ are Lipschitz continuous.

%The paper has changed substantially from the initial submission. Please use the publisher's version.
This can be stated as follows: \\
Let us consider the following Dirichlet boundary problem\vspace{-2mm}
\be\label{eqHypDir}\left\{
\begin{array}{ll}
 u(x)=\phi(x),        & \hbox{if} \,\, x\in \mathbb{S}^{n-1}, \\
(\Delta_h)u(x)=\psi, & \hbox{if}  \,\,  x\in \mathbb{B}^n .
\end{array}
\right.\vspace{-2mm}\ee
%%We use that
%$g_h(r)=O((1-r)^{n-1})$, $g_h(x,y)= g(|T_y x|)$, where   $g_h$ is Green %hyperbolic.
%Theorem 1.1
\begin{theorem}\cite{ChenRas}\label{ChenRas1}
Suppose that $u \in  C^2(\mathbb{B}^n , \mathbb{R}^n ) \cap C(\overline{\mathbb{B}^n} , \mathbb{R}^n)$ for  $n\geqslant 3$ and\vspace{-3mm}
$$\int_{\mathbb{B}^n}(1 {-} |x|^2)^{n{-}1}|\psi(x)| \dd\tau (x) {\leqslant}\mu_1, \mbox{ where }\mu_1 > 0\mbox{ is a constant. }\hspace{1cm}\vspace{-2mm}$$
%(1.2)
If u satisfies (\ref{eqHypDir}), then\vspace{-2mm}
\noindent\begin{enumerate}
\item[$(1)$] $u = P_h [\phi] - G_h [\psi]$   and\vspace{-2mm}
\item[$(2)$] $U=u\circ \varphi_x=  P_h [\phi\circ \varphi_x] - G_h [\psi\circ \varphi_x]$,  $x \in \mathbb{B}^n$ .
\end{enumerate}
\end{theorem}

%Suppose\vspace{-1mm}
%\begin{itemize}
%  \item[(h1)]   $u\in C^{2}(\mathbb{B}^{n},\IR^n) \cap
%C(\overline{\mathbb{B}^{n}},\IR^n )$ is of the form
%(1),\vspace{-1mm}
%  \item[(h2)]  $\phi$  L-Lip   on  $\mathbb{S}^{n-1}$,\vspace{-1mm}
%  \item[(h3)]  there is a constant
%$M\geqslant 0$  such that
% $|\psi(x)|\leqslant  M(1-|x|^2)$,
 %\item[(h3-1)] $(1-|y|^2)^{-2} \psi $ belongs  $L^p$ for some $p>n$.
%\end{itemize}
%Then there exist constants
%$C_1=C_1(n,L,M)$
%and
%$C_2=C_2(n,\phi,\psi)$
%such that  $u$ is   $C_1,C_2$   Bi-Lip  on  $\mathbb{B}^{n}$.

%For example,  for  $\psi=0$,  $u=P_{h}[\phi]$  is  Bi-Lip?

%Question.  partial derivatives   of  $P_{h}[\phi]$ are  bounded on   $Lip(\mathbb{S}^{n-1})$?  It is true.

In \cite{ChenRas} the following result is also established:
\begin{theorem}\label{ChenRasThm}\cite[Theorem 1.2]{ChenRas}
  Let $n\geq 3$. Suppose that\vspace{-2mm}
  \begin{itemize}
    \item[(1)] $u\in C^2(\mathbb{B}^n,\mathbb{R}^n)\cap C(\overline{\mathbb{B}^n},\mathbb{R}^n)$ is of the form (\ref{RepHyp});\vspace{-2mm}
    \item[(2)] there is a constant $L\geqslant 0$ such that $|\phi(\xi)-\phi(\eta)|\leqslant L|\xi-\eta|$ for all $\xi,\eta\in\mathbb{S}^{n-1}$;\vspace{-2mm}
    \item[(3)] there is a constant $M\geqslant 0$ such that $|\psi(x)|\leqslant M(1-|x|^2)$ for all $x\in\mathbb{B}^n$.\vspace{-2mm}
  \end{itemize}
  Then, there is a constant $N=N(n,L,M)$ such that for $x,y\in\mathbb{B}^n$,\vspace{-2mm}
  $$|u(x)-u(y)|\leqslant N|x-y|,\vspace{-2mm}$$
  where the notation $N=N(n,L,M)$ means that the constant $N$ depends only on the quantities $n, L$ and $M$.\vspace{-3mm}
\end{theorem}
In fact, they prove more general result: \vspace{-3mm}
\begin{itemize}
\item[(A)] If function $\phi$ satisfies $(1)$ then $\Phi=P_h[\phi]$ is Lipshitz and\vspace{-3mm}
\item[(B)] if function $\psi$ satisfies $(3)$ then $\Psi=G_h[\psi]$ is Lipshitz.\vspace{-3mm}
\end{itemize}

Let us note that statement corresponding to $(A)$ is not valid for the euclidean Poisson's integral. For more details, see below.

In Proposition \ref{MainProp} we give a brief proof of the above theorem. In \cite[Lemma 5.2]{ChenRas} the authors used some hypergeometric series techniques for proving inequalities $(\ref{cl4})$ and $(\ref{cl5})$. Instead of this approach we used Proposition \ref{MMThmSphere} and some basic inequalities, which are modification of some estimates from planar Hardy space theory (for details see Lemma \ref{IntLem1}, \ref{IntLem2} and Lemma \ref{mainProp1}), where the simple change of variables\vspace{-2mm}
\begin{equation*}
\theta=(1-r) u,\vspace{-2mm}
\end{equation*}
in the integral \vspace{-2mm}
$$\int\limits_0^{\infty}\frac {\theta^{\alpha +n-2}}
{\left((1-r)^2+\frac{4r}{\pi^2}\, \theta^2\right)^{n-1}}\,d\theta,$$
gives optimal growth estimates, for our purposes.

In Theorem \ref{thmloc0}, it is proved, that local $\alpha-$H\"{o}lder continuity ($0<\alpha\leqslant 1$) at the boundary point $x$ of $\mathbb{B}^n$ implies H\"{o}lder continuity along the whole radius of that point. This result is analogous to \cite[Theorem 6.2]{mss}, except for the case $\alpha=1$. The case of Lipshitz boundary function $\phi$ for harmonic functions is investigated by the first author with M. Arsenovi\'{c} and V. Mabojlovi\'{c} in  \cite{MAVMMM}. Under additional condition that $P[\phi]$ is $K$-quasiregular, they proved that $P[\phi]$ is Lipshitz on $\mathbb{B}^n$.

% which can be included in the case of hyperbolic harmonic functions.

In \cite{HypHar2} the authors proved that condition $(3)$ can't be excluded from the statement of the Theorem \ref{ChenRasThm}. In subsection 3.1 we introduced some alternative assumptions on hyperbolic Laplacian $\psi$ in order to get H\"{o}lder or Lipshitz continuity of hyperbolic Green potential of $\psi$. Euclidean Green potential, among other things, is investigated by the first author in \cite{Mat}.

In Section 5, we proved, when we assume $C_c^1$ smoothness of boundary  functions $f$, that all partial derivatives of hyperbolic Poisson's integral of $f$ can be continuously extended to the boundary of $\mathbb{H}^n$. At first, we show that, if $u$ is  hyperbolic harmonic in the upper half-space, then $\frac{\partial u}{\partial y}(x_0,y)\to 0, y\to 0^+$, when boundary function $f$ of the functions $u$ is differentiable at the boundary point $x_0$. Generally, this is not true for (Euclidean) harmonic functions. In \cite{Pri} we can see an example of $g\in C^1(\mathbb{R})$ such that (Euclidean) harmonic extension u of $g$ in $\mathbb{H}^2$ satisfies $\limsup\limits_{ y\to 0^+}\frac{\partial u}{\partial y}(0,y)= +\infty$.

%This is not true for harmonic?

%$(h_2)$   $\phi$  is $C^1$    on  $\mathbb{S}^{n-1}$.

%Question 2.  If we replace  $(2)$ with $h_2$, $u$ is   $C^1$ on   $\overline{\mathbb{B}^{n}}$?

%%%%%%%%%%%%%%%%%%%%%%%%%%%%%%%%%%%%%%%%%%%%%%%%%%%%%%

%In a similar way we can consider $T_\alpha$-harmonic functions  for    $n-2\geq\alpha$   and   $n-2\leq\alpha$ .
%We can also use a variation of the above proof.

%For example   in the planar case   if  $\varphi_a(z)= \dfrac{a-z}{1-\overline{a} z}$, then  $\varphi_a'(z)= \dfrac{|a|^2-1}{(1-\overline{a} z)^2}$.
\smallskip

\section{Dirichlet problem for Hyperbolic Poisson's equation}

For Harmonic and Subharmonic Function Theory on the Hyperbolic Ball we refer the interested reader to Stoll \cite{stoll}.

\begin{defin}[Definition 4.3.1 \cite{stoll}]
  Let $D$ be a subset of $\mathbb{R}^n$. A function $f:D\to[-\infty,+\infty)$ is \textit{upper semicontinuous} at $x_0\in\ D$ if for every $\alpha\in\mathbb{R}^n$  with $\alpha>f(x_0)$ there exists a $\delta>0$ such that\vspace{-1mm}
  $$f(x)<\alpha \mbox{ for all } x\in D\cap B(x_0,\delta).\vspace{-2mm}$$
\end{defin}
\begin{defin}[Definition 4.3.3 \cite{stoll}]
  Let $\Omega$ be an open subset of $\mathbb{B}^n$. An upper semicontinuous function $f:\Omega\to [-\infty,+\infty)$, with $f\not\equiv -\infty$, is $\mathcal{H}$-\textit{subharmonic} on $\Omega$ if \vspace{-1mm}
  $$f(a)\leqslant\int_{\mathbb{S}^{n-1}}f(\varphi_a(rt))\dd\sigma(t)\vspace{-1mm}$$
  for all $a\in\Omega$ and all $r$ sufficiently small.
\end{defin}
\begin{theorem}[Theorem 4.6.3 \cite{stoll}]
  If $f$ is a $\mathcal{H}$-subharmonic on $\mathbb{B}^n$, then there exsits a unique regular Borel measure $\mu_f$ on $\mathbb{B}^n$ such that\vspace{-1mm}
  \begin{equation}\label{weak-hyp}
  \int_{\mathbb{B}^n}\psi\dd\mu_f=\int_{\mathbb{B}^n}f\Delta_h\psi\dd\tau\vspace{-1mm}
  \end{equation}
  for all $f\in C_c^2(\mathbb{B}^n)$.
\end{theorem}

\begin{defin}[Definition 4.6.4 \cite{stoll}]
  If $f$ is a $\mathcal{H}$-subharmonic on $\mathbb{B}^n$, the unque regular Borel measure $\mu_f$ satisfying $(\ref{weak-hyp})$ is said to be the \textbf{Riesz measure} of $f$.
\end{defin}
\begin{defin}
  As in the Euclidean case, for $\zeta\in\mathbb{S}^{n-1}$ and $\alpha>1$, we denote by $\Gamma_\alpha(\zeta)$ the \textit{non-tangental approach region} at $\zeta$ defined by\vspace{-1mm}
  $$\Gamma_\alpha(\zeta)=\{y\in\mathbb{B}^n:|y-\zeta|<\alpha(1-|y|)\}.\vspace{-1mm}$$
\end{defin}
\begin{theorem}[Theorem 8.3.3 \cite{stoll}]
  \begin{itemize}
  \item[]
    \item[(a)] If $f\in L_1(\mathbb{S}^{n-1})$, then for every $\alpha>1$,\vspace{-1mm}
        $$\lim\limits_{\substack{x\to \zeta,\\ x\in\Gamma_\alpha(\zeta)}}P_h[f](x)=f(\zeta)\mbox{ $\sigma$-a.e. on } \mathbb{S}^{n-1}.\vspace{-1mm}$$
    \item[(b)] If $\nu$ is a signed Borel measure on $\mathbb{S}^{n-1}$ which is singular w.r.t $\sigma$, then for every $\alpha>1$,\vspace{-1mm}
        $$\lim\limits_{\substack{x\to \zeta,\\ x\in\Gamma_\alpha(\zeta)}}P_h[f](\nu)=0 \mbox{ $\sigma$-a.e. on } \mathbb{S}^{n-1}.\vspace{-1mm}$$
  \end{itemize}
\end{theorem}
We say that  a positive measure  $\mu$  satisfies integrability condition if
\begin{equation}\label{intCondMu}
  \int_{\mathbb{B}^{n}}(1-|y|^{2})^{n-1} d\mu(y) < \infty.
\end{equation}

We recall that Green potential of regular Borel measure $\mu$ on $\mathbb{B}^n$ is defined as (see \cite{stoll}) $G_{\mu}(x)=\int_{\mathbb{B}^n}G_h(x,y)\dd\mu(y)$. In \cite{stoll} it is proved that $G_\mu(x)\not\equiv+\infty$ if and only if $(\ref{intCondMu})$ holds.

\begin{propo}[Corrolary 4.1.5 \cite{stoll}]\label{gr-pot-psi}
If $f\in C_c^2(\mathbb{B}^n)$, then for all $a\in\mathbb{B}^n$,\vspace{-1mm}
$$f(a)=-\int_{\mathbb{B}^n}G_h(a,x)\Delta_h f(x)\dd\tau(x).$$
\end{propo}

\begin{theorem}[Theorem 9.4.1 \cite{stoll}] Let $G_\mu$ be the Green potential of a
measure $\mu$  satisfying $(\ref{intCondMu})$. Then
$$\lim\limits_{r\rightarrow 1} G_\mu(rt) = 0 \mbox{ for almost every } t \in  \mathbb{S}^{n-1}.$$
\end{theorem}

The following theorem can be regarded as a generalisation of the Theorem \ref{ChenRas1} from \cite{ChenRas}, which is Dirichlet problem for hyperbolic Poisson's equation.
\bthm  Suppose  that a Borel nonnegative measure  satisfies the
integrability condition $(\ref{intCondMu})$ and that   $u=P_{h}[\phi]-G_{h}[\mu]$, where
$\phi\in L^1( \mathbb{S}^{n-1})$. Then $(\Delta_h)u =\mu$ in weak sense, and   $$\lim\limits_{r\rightarrow 1}
u(rt) = \phi(t) $$for almost every $t \in  \mathbb{S}^{n-1}$.
\ethm
%  Suppose  that  the Riesz measure of $f$,  $\mu=\mu_f$

\begin{proof}
Let $v(x)=-G_\mu(x)$ and $\psi\in C_c^{\infty}(\mathbb{B}^n)$. Then\vspace{-1mm}
  $$\int_{\mathbb{B}^n}v(x)\Delta_h\psi(x)\dd\tau(x)=-\int_{\mathbb{B}^n}\left(\int_{\mathbb{B}^n}G_h(x,y)\dd\mu(y)\right)\Delta_h\psi(x)\dd\tau(x).\vspace{-1mm}$$
  By Fubini theorem, we have that\vspace{-1mm}
  $$\int_{\mathbb{B}^n}v(x)\Delta_h\psi(x)\dd\tau(x)=-\int_{\mathbb{B}^n}\left(\int_{\mathbb{B}^n}G_h(x,y)\Delta_h\psi(x)\dd\tau(x)\right)\dd\mu(y).\vspace{-1mm}$$
  Using Proposition \ref{gr-pot-psi} we get\vspace{-1mm}
  $$\int_{\mathbb{B}^n}v(x)\Delta_h\psi(x)\dd\tau(x)=\int_{\mathbb{B}^n} \psi(y)\dd\mu(y).\vspace{-1mm}$$
  For the second part of this proof use Theorem 8.3.3 subsection 5.9 Theorem 2  and  Theorem 9.4.1
from \cite{stoll}.
\end{proof}

\section{Hyperbolic harmonic functions in the unit ball and local H\"{o}lder continuity}

Firstly, we refer to the paper \cite{mss}, specifically, to the Proposition 5.10. In this proposition we use the following notation: $\sigma_{n-1}$ is the surface area of the sphere $\mathbb{S}^{n-1}$ and $\varphi$ is an angle between radius vector of point $\eta\in\mathbb{S}^{n-1}$ and radius vector of the point $\tilde{x}.$

Let us define $\sigma_*(n)=\frac{\sigma_{n-2}}{\sigma_{n-1}}$. Using formula $\sigma_{n-1}=\frac{2\pi^{n/2}}{\Gamma(\frac{n}{2})}$ we get $\sigma_*(n)=\frac{1}{\sqrt{\pi}}\frac{\Gamma(n/2)}{\Gamma((n-1)/2)}$.

\begin{propo}(Proposition 5.10 \cite{mss})\label{MMThmSphere}
  If $f$ is a function on $\mathbb{S}^{n-1}$ depending only on $\varphi$, then\vspace{-1mm}
  $$\int_{\mathbb{S}^{n-1}}f(\eta)\,\mathrm{d}\sigma(\eta)=\sigma_{n-2}\int_{0}^{\pi}f(\varphi)\sin^{n-2}\varphi\,\mathrm{d}\varphi.$$
\end{propo}

    We will prove the following theorem, which is analogous to \cite[Theorem 6.2]{mss}.

\begin{thm}\label{thmloc0}
Suppose  that $0 < \alpha\leqslant1$,    $\,h\,$ is    a hyperbolic harmonic
mapping from $\mathbb{B}^n$ which is
continuous on  $\overline{{\B}^{n}}$, and\smallskip

\noindent (h1)  let $x_0\in \mathbb{S}^{n-1}$  and   \ $ |h(x)-h(x_0)|
\leqslant M |x-x_0|^\alpha$\  \ for\ \  $x\in \mathbb{S}^{n-1}$.\smallskip

Then there is a constant   $M_n$   such that\vspace{-2mm}
$$(1-r)^{1-\alpha} |h'(r x_0)|\leqslant M_n,\quad 0\leqslant  r < 1.\vspace{-1mm}$$
\end{thm}
\begin{proof}
Let  $h_b$ denote  the restriction of $h$ on   $\mathbb{S}^{n-1}$.
%By hypothesis (h1)  $h_b$  is  $\alpha$ -H\"{o}lder at $x_0 \in \mathbb{S}^{n-1}$.
%belongs to ${\rm Lip}\,(\alpha)$.
Since   $h$ is hyperbolic harmonic on $\mathbb{B}^n$  and continuous on
$\overline{\B^{n}}$, then\vspace{-1mm}
\begin{equation}\label{RepHar1}
h(x)=\int\limits_{\mathbb{S}^{n-1}} P_h(x,\eta) h_b(\eta)
\dd\sigma(\eta)\vspace{-1mm}
\end{equation}
for every $x\in \mathbb{B}^n$. Set  $d:=d(x)= 1  - |x|^2$.  By computation\vspace{-2mm}
$$\partial_{x_k} P_h(x,t)= -2(n-1)\left(\frac{x_k}{|x-t|^2} + d(x)
\frac{x_k-t_k}{|x-t|^4}\right)\left(\frac{1 - |x|^2}{|t - x|^2}\right)^{n-2}.\vspace{-2mm}$$
Hence,  if  $d\leqslant |x-t|$,  then\vspace{-2mm}
\begin{equation}\label{(1)}
 |\partial_{x_k} P_h(x,t)|\leqslant c_1 \frac{(1-|x|^2)^{n-2}}{|x-t|^{2(n-1)}}.\vspace{-2mm}
\end{equation}

Let  $x=r e_n$  and let $\theta$ be the angle between  $t$  and $e_n$.
Then $s:= |x-t|^2= 1 -2r \cos \theta + r^2$ depends only on $\theta$
for fixed  $x$.
% est.loc1
Next,  since  $\int\limits_{\mathbb{S}^{n-1}} \partial_k P_h(x,t)
h(e_n) \dd\sigma(t)=0$, we find\vspace{-3mm}
\begin{equation}\label{}
\partial_{x_k}h(x) = \int\limits_{\mathbb{S}^{n-1}} \partial_k P_h(x,t) \big(h(t) -h(e_n)\big)  \dd\sigma(t).\vspace{-2mm}
\end{equation}
Hence by  (\ref{(1)}) and the hypothesis  $(h1)$, we get\vspace{-2mm}
\begin{equation}\label{est.loc0}
|\partial_{x_k}h(x)| \leqslant c_2 (1-|x|^2)^{n-2}\int\limits_{\mathbb{S}^{n-1}}
\frac{|e_n -t|^\alpha}{|x-t|^{2(n-1)}} \dd\sigma(t).\vspace{-2mm}
\end{equation}
Therefore  the   proof of Theorem  \ref{thmloc0} is reduced to the
proof of the following proposition.
\end{proof}

\begin{prop}\label{prop2a}
Suppose  that $n\geqslant 3$ and $0 <\alpha\leqslant 1$  and $x=re_n$ for $0<r<1$. Then\vspace{-1mm}
$$I_\alpha(r e_n):= (1-|x|^2)^{n-2}\int\limits_{\mathbb{S}^{n-1}} \frac{|e_n -t|^\alpha}{|x-t|^{2(n-1)}}
\dd\sigma(t)\leqslant c\cdot \frac 1{(1-r)^{1-\alpha}},\vspace{-2mm}
$$
%for  $\frac 12 \leqslant r <1$.
where   $c=c(\alpha,n)$  is a positive constant which depends only
on  $n$ and $\alpha$.
\end{prop}
\smallskip

Using similar  approach,  if  $\omega$ is a majorant and $\delta_r=1-r^2$, one can prove\vspace{-1mm}
$$I_\omega(r e_n):= (1-|x|^2)^{n-2}\int\limits_{\mathbb{S}^{n-1}} \frac{\omega(|e_n -t|)}{|x-t|^{2(n-1)}}
\dd\sigma(t)\leqslant c\cdot \frac {\omega(\delta_r)} {\delta_r}.\vspace{-1mm}
$$

\begin{proof}
We use        spherical cups    $S^{\theta}$  defined  by  $t_n  >
\cos \theta$  and  integration with parts. Since for a fixed
$\theta\in [0,\pi]$, $|e_n -t|\leqslant \theta$  for  $t\in
S^{\theta}$, by  an application of    Proposition \ref{MMThmSphere} to
$f(t)=\frac{|e_n -t|^\alpha}{|x-t|^{2(n-1)}}$, we get
%(see also Remark\ref{Rm1} below)
\vspace{-4mm}
\begin{eqnarray}\label{eqH4}
I_\alpha(r e_n) \leqslant c_3 (1-|x|^2)^{n-2}\int\limits_0^\pi \frac
{|\theta|^{n-2}
|\theta|^{\alpha}} {((1-r)^2+\frac{4r}{\pi^2}\theta^2)^{n-1}}\,d\theta \nonumber \\
\label{eqH5} <c_4(1-|x|^2)^{n-2}\int\limits_0^{\infty}\frac {\theta^{\alpha +n-2}}
{\left((1-r)^2+\frac{4r}{\pi^2}\, \theta^2\right)^{n-1}}\,d\theta\,.
\end{eqnarray}
Next using  $(1+\frac{4r}{\pi^2} u^2)^{-1} \leqslant c_5
(1+u^2)^{-1}$  for  $\frac 12 \leqslant r <1$  and  a  change of
variables \vspace{-1mm}
\begin{equation}\label{priv}
\theta=(1-r) u,\vspace{-2mm}
\end{equation}
  we find
\begin{equation}
I_\alpha(r e_n) \leqslant c_6 {(1-r)^{\alpha -1
}}\int\limits_0^{\infty}\frac{u^{\alpha +n-2}}{(1+u^2)^{n-2}}\,d
    u.\vspace{-2mm}
\end{equation}
Denote by $J(\alpha)$   the last expression  on the right hand side
of the previous formula. Since  $g(u)= \frac{u^{\alpha
+n-2}}{(1+u^2)^{n-1}}\sim u^{\alpha-n}$ for   $u\rightarrow  +
\infty$ and  by  hypothesis   $0 <\alpha\leqslant1$  (and therefore
$\alpha-n <-1$),  the integral $J(\alpha)$  converges  and we have
\vspace{-2mm}
\begin{equation}\label{IntEst1}
  I_\alpha(r e_n) \leqslant c_7 (1-r)^{\alpha -1}\quad \mbox{for}\quad  \frac
12 \leqslant r <1.\vspace{-2mm}
\end{equation}

%the proof follows.
%$(1+\frac{4r}{\pi^2} u^2)^{-1} \leqslant c_5 (1+u^2)^{-1}$

$(1-r)^{1-\alpha} A(r)$ is continuous on  $[0,1/2]$ and attains a
maximum $c_8$, that is
\vspace{-2mm}
\begin{equation}\label{IntEst2}
    I_\alpha(r e_n) \leqslant c_9 (1-r)^{\alpha -1}\quad \mbox{for}\quad 0
\leqslant r \leqslant  \frac 12, \vspace{-2mm}
\end{equation}
where  $c_9= c_3c_8 $.
Hence  from (\ref{IntEst1}) and (\ref{IntEst2}) with $c=\max \{c_7,c_9 \}$  the proof of Proposition follows.
\end{proof}

On this point, it is interesting to observe that case $\alpha=1$, in other words, Lipschitz continuity of boundary function, does not imply Lipschitz continuity of its Euclidean harmonic extension. However, in the case of hyperbolic harmonic extension, under the assumptions of Theorem \ref{thmloc0}, we can conclude that, when $\alpha=1$, the statement of this Theorem holds.

\medskip

Combining Proposition~\ref{prop2a} and (\ref{est.loc0}) we get the proof of Theorem.

\medskip
%\begin{thm}\label{thmloc0'}
%Suppose  that $0 < \alpha\leqslant 1$,    $\,h\,$ is    hyperbolic  harmonic %mapping from $\mathbb{B}^n$ which is
%continuous on  $\overline{{\B}}^{n}$, and\smallskip

%\noindent (h1)  let $x_0\in \mathbb{S}^{n-1}$  and   $ |h(x)-h(x_0) |
%\leqslant M |x-x_0|^\alpha$ for $x\in \mathbb{S}^{n-1}$.\smallskip

%Then there is a constant   $M_n$   such that\vspace{-2mm}
%$$ \mathrm{(S1)}\quad (1-r)^{1-\alpha} |h'(r x_0)|\leqslant M_n,\quad %0\leqslant  r <1.\hspace{8cm}$$
%\end{thm}

\section{Growth of partial derivatives for the hyperbolic Green potential}

In this section, we consider H\"{o}lder continuity and Lipschitz continuity of hyperbolic Green potential on the unit ball in $\mathbb{R}^n$.

We need the following statement.
\vspace{-3mm}
\begin{lem}\label{IntLem1}
  Let $n\geqslant 3$ and $r=|x|$, where $0<r<1$. If\vspace{-1mm} $$A(r,\rho)=\int_{\Sp^{n-1}}\frac{\mathrm{d}\sigma(\xi)}{\sqrt{1-2\rho\langle x,\xi\rangle+\rho^2|x|^2}}\quad \mbox{and}\quad
  B(r,\rho)=\int_{\Sp^{n-1}}\frac{\mathrm{d}\sigma(\xi)}{1-2\rho\langle x,\xi\rangle+\rho^2|x|^2},\vspace{-1mm}$$
  then we have the following conclusions.
  \begin{itemize}
  \item[(i)] There exists $C_1, C_2>0$ such that for all $0<\rho<1$ and $1/2<r<1$ we have that
  $A(r,\rho)\leqslant \frac{C_1(n)}{\sqrt{\rho}}$ if $n\geqslant 3$ and $B(r,\rho)\leqslant \frac{C_2(n)}{\rho}$  if $n>3$.
  \item[(ii)] If $n=3$ then there exist $C_3, C_4>0$ such that $B(r,\rho)\leqslant C_3-C_4\log(1-\rho)$ for $0<\rho<1$ and $1/2<r<1$.
  \end{itemize}
\end{lem}
\bpf
By using Proposition \ref{MMThmSphere}, we can check that\vspace{-1mm}
$$A(r,\rho){=}\sigma_*(n)\int_{0}^{
\pi}\frac{\sin^{n{-}2}\theta}{\sqrt{(1{-}\rho r)^2{+}4\rho r\sin^2\frac{\theta}{2}}}\,\mathrm{d}\theta,\quad
B(r,\rho){=}\sigma_*(n)\int_{0}^{
\pi}\frac{\sin^{n{-}2}\theta}{(1{-}\rho r)^2{+}4\rho r\sin^2\frac{\theta}{2}}\,\mathrm{d}\theta,\vspace{-2mm}$$
where $\theta$ is an angle between vectors $x$ and $\xi$. Using $\sin x\geqslant \frac{2}{\pi}x$ for  $x\in (0,\frac{\pi}{2})$, we get\vspace{-2mm}
$$A(r,\rho)\leqslant \frac{\widetilde{C}_1}{\sqrt{\rho r}}\int_{0}^{\pi}\frac{\theta^{n-2}}{\theta}\,\mathrm{d}\theta,\quad B(r,\rho)\leqslant\frac{\widetilde{C}_2}{\rho r}\int_{0}^{\pi}\frac{\theta^{n-2}}{\theta^2}\,\mathrm{d}\theta.\vspace{-1mm}$$
The first and the second integral converges in case $n\geqslant 3$ and $n\geqslant 4$, respectively, which gives us $(i)$.
In order to prove $(ii)$ let us assume $n=3$. Then, we have \vspace{-1mm}
$$B(r,\rho)\leqslant\sigma_*(3)\int_{0}^{\pi}\frac{\theta\,\mathrm{d}\theta}{(1-\rho)^2+\frac{4r\rho}{\pi^2}\theta^2}.\vspace{-1mm}$$
In the case of $0<\rho<1/2$ we have $B(r,\rho)\leqslant 2\pi^2\sigma_*(3)$. Let $1/2<r$, $\rho<1$ and let's use the change of variables $\theta=(1-\rho)u$. Then we have \vspace{-2mm}
\begin{align*}
B(r,\rho)& \leqslant\sigma_*(3)\int_{0}^{\frac{\pi}{1-\rho}}\frac{u\,\mathrm{d}u}{1+\frac{u^2}{\pi^2}} =\sigma_*(3)\left(\int_{0}^{1}\frac{u\,\mathrm{d}u}{1+\frac{u^2}{\pi^2}}+\int_{1}^{\frac{\pi}{1-\rho}}\frac{u\,\mathrm{d}u}{1+\frac{u^2}{\pi^2}}\right)\\
& \leqslant C_3+ \widetilde{C}_4\int_{1}^{\frac{\pi}{1-\rho}}\frac{u\,\mathrm{d}u}{u^2}\\
& \leqslant C_3-C_4\log(1-\rho).
\end{align*}
\epf\vspace{-5mm}

\begin{prop}\label{mainProp1}
Suppose $n\geqslant 3,\psi\in C(\mathbb{B}^n,\mathbb{R}^n)$ and $|\psi(x)\leqslant M(1-|x|^2)$ in $\mathbb{B}^n$, where $M$ is a constant. Let\vspace{-1mm}
$$I_{2,k}(x)=\int_{\B^n}\left| \frac{\partial}{\partial x_k}G_h(x,y)\psi(y)\right|\,\mathrm{d}\tau(y).$$
Then there is a constant $\beta=\beta(n,M)$ such that
$$I_{2,k}(x)\leqslant\beta_1,\mbox{ for all } x\in\mathbb{B}^n.$$
\end{prop}
\bpf
Formula (5.3) from \cite{ChenRas} gives $\frac{\partial}{\partial x_k}G_h(x,y)= (D_k G_{h})_1(x,y) + (D_k G_{h})_2(x,y)$, where
%\begin{align*}
%  \frac{\partial}{\partial x_k}G_h(x,y)
%  = & -\frac{(x_k-y_k)(1-|x|^2)^{n-1}(1-|y|^2)^{n-1}}{n|x-y|^n[x,y]^n}\\\vspace%{-2mm}
%  & - \frac{x_k(1-|x|^2)^{n-2}(1-|y|^2)^{n-1}}{n|x-y|^{n-2}[x,y]^n}.
%\end{align*}
%\vspace{-2mm}
\vspace{-2mm} $$(D_k G_{h})_1(x,y):= -\frac{(x_k-y_k) (1-|x|^2)^{n-1} (1-|y|^2)^{n-1}}{n|x-y|^n[x,y]^n}$$
\vspace{-2mm}
and  $$(D_k G_h)_2(x,y):=- \frac{x_k(1-|x|^2)^{n-2}(1-|y|^2)^{n-1}}{n|x-y|^{n-2}[x,y]^n}.\vspace{-1mm}$$
\vspace{-2mm}
Also, from \cite{ChenRas} we have\vspace{-2mm}
$$I_{2,k}(x)\leqslant \frac{1}{n}\left(I_{3,k}+I_{4,k}\right),\vspace{-1mm}$$
where\vspace{-2mm}
$$I_{3,k}(x)=\int_{\B^n}\frac{|x_k-y_k|(1-|x|^2)^{n-1}(1-|y|^2)^{n-1}}{|x-y|^n[x,y]^n}|\psi(y)|\,\mathrm{d}\tau(y)\vspace{-1mm}$$
and\vspace{-1mm}
$$I_{4,k}(x)=\int_{\B^n}\frac{|x_k|(1-|x|^2)^{n-2}(1-|y|^2)^{n-1}}{|x-y|^{n-2}[x,y]^n}|\psi(y)|\,\mathrm{d}\tau(y).\vspace{-1mm}$$

After applying condition (h3), and introducing change of variable $y=\varphi_x(w)$, having in mind (\ref{IdentMob1}) and (\ref{IdentMob2}),
we get
\begin{equation*}
    |I_{3,k}(x)|\leqslant J_{3,k}(x):=M\int_{\mathbb{B}^n}\frac{\dd\nu(w)}{[x,w]|w|^{n-1}} \mbox{ and } |I_{4,k}(x)|\leqslant J_{4,k}(x):=M\int_{\mathbb{B}^n}\frac{\dd\nu(w)}{[x,w]^2|w|^{n-2}}.
\end{equation*}
(For more details, see \cite[$Claim\ 5.3,5.4$]{ChenRas}). Finally, using Lemma \ref{IntLem1}, we get
\begin{equation}\label{cl4}
J_{3,k}(x)= n M\int_{0}^{1}\int_{\Sp^{n-1}}\frac{\mathrm{d}\sigma(\xi)}{[x,\rho\xi]}\,\mathrm{d}\rho=\int_{0}^{1}A(r,\rho) \,\mathrm{d}\rho<+\infty\vspace{-2mm}
\end{equation}
and\vspace{-2mm}
\begin{equation}\label{cl5}
J_{4,k}(x)=n M\int_{0}^{1}\rho\,\mathrm{d}\rho\int_{\Sp^{n-1}}\frac{\mathrm{d}\sigma(\xi)}{[x,\rho\xi]^2}=\int_{0}^{1}\rho B(r,\rho)\,\mathrm{d}\rho<+\infty.\vspace{-1mm}
\end{equation}
Here, we used that $[x,\rho\xi]=\sqrt{1-2\rho\langle x,\xi\rangle+\rho^2|x|^2}$.
\epf
Let $\Omega\subset\mathbb{R}^n$ and $F,G:\Omega\rightarrow\mathbb{R}_+:=[0,+\infty)$. We use notations $F(x)\preceq G(x), x\in\Omega$ which means that there is $C>0$ such that $F(x)\leqslant CG(x), x\in\Omega$.
\begin{lem}
  Suppose that $n\geqslant 3$. Let $I_m=\int_{\Sp^{n-1}}\frac{\mathrm{d}\sigma(\xi)}{|x-\xi|^m}$, where $x\in\B^n$ and $r=|x|$. When $r\rightarrow 1^-$, we have\vspace{-2mm}
  \begin{itemize}
    \item[(i)] $I_m$ is bounded for $0<m<n-1$.\vspace{-2mm}
    \item[(ii)] $I_{n-1}\preceq\log\frac{1}{1-r}$.\vspace{-2mm}
    \item[(iii)] $I_m\preceq\frac{1}{(1-r)^{m-n+1}}$ for $m>n-1$.\vspace{-2mm}
  \end{itemize}
\end{lem}
\bpf
Using Proposition \ref{MMThmSphere} we can conclude that,\vspace{-2mm}
$$I_m=\sigma_*(n)\int_{0}^{\pi}\frac{\sin^{n-2}\theta}{((1-r)^2+4r\sin^2\frac{\theta}{2})^{\frac{m}{2}}}\,\mathrm{d}\theta.\vspace{-2mm}$$
Now we can rewrite integral $I_m$ in the form\vspace{-2mm}
\begin{align*}
  I_m & = \sigma_*(n)\left(\int_{0}^{\frac{\pi}{2}}\frac{\sin^{n-2}\theta}{((1-r)^2+4r\sin^2\frac{\theta}{2})^{\frac{m}{2}}}\,\mathrm{d}\theta+\int_{\frac{\pi}{2}}^{\pi}\frac{\sin^{n-2}\theta}{((1-r)^2+4r\sin^2\frac{\theta}{2})^{\frac{m}{2}}}\,\mathrm{d}\theta\right).\vspace{-3mm}
\end{align*}
Since $\sin x\sim x$ for $x\in (0,\pi/2)$ we have that, when $r\ra 1^-$, there is $c_3>0$ such that\vspace{-2mm}
$$I_m\preceq c_3\int_{0}^{\frac{\pi}{2}}\frac{\theta^{n-2}}{\theta^m}\dif\theta,\vspace{-2mm}$$
which gives us conclusion $(i)$. Also, there exists $c_4,c_5>0$ such that
$$I_m=c_4\int_{0}^{\frac{\pi}{2}}\frac{\sin^{n-2}\theta}{((1-r)^2+4r\sin^2\frac{\theta}{2})^{\frac{m}{2}}}\,\mathrm{d}\theta+c_5.$$
In order to prove $(ii)$, we use change of variables $\theta=(1-r)u$. Then we have
\begin{align*}
I_m \preceq\frac{1}{(1-r)^{m-n+1}}\int_{0}^{\frac{\pi}{2(1-r)}}\frac{u^{n-2}\,\mathrm{d}u}{(1+u^2)^{\frac{m}{2}}},
\end{align*}
which gives us our conclusion Now, let us assume that $m=n-1$. We have
\begin{align*}
  I_{n-1} & \preceq \int_{0}^{\frac{\pi}{2(1-r)}}\frac{u^{n-2}\,\mathrm{d}u}{(1+u^2)^{\frac{n-1}{2}}}\\
   & = \left(\int_{0}^{1}\frac{u^{n-2}\,\mathrm{d}u}{(1+u^2)^{\frac{n-1}{2}}}+\int_{1}^{\frac{\pi}{2(1-r)}}\frac{u^{n-2}\,\mathrm{d}u}{(1+u^2)^{\frac{n-1}{2}}}\right)\\
   & \preceq \log\frac{1}{1-r}.
\end{align*}
\epf
Before the next result, let us prove the following
\begin{lem}\label{IntLem2}
If $r=|x|$, $0\leqslant\rho< 1,1/2\leqslant r<1$ and $B(r,\rho)$ is defined as in Lemma \ref{IntLem1} we have that
\vspace{-2mm}
\begin{equation}\label{BInt}
    \int_0^1 B(r,\rho)\dd\rho< +\infty
\end{equation}
\end{lem}
\begin{proof}
By using Proposition \ref{MMThmSphere}, we can check that\vspace{-1mm}
$$B(r,\rho){=}\sigma_*(n)\int_{0}^{
\pi}\frac{\sin^{n{-}2}\theta}{(1{-}\rho r)^2{+}4\rho r\sin^2\frac{\theta}{2}}\,\mathrm{d}\theta.\vspace{-2mm}$$
If $0<\rho\leqslant 1/2$, we have\vspace{-2mm}
$$B(r,\rho)\preceq \int_0^{\pi}\frac{\theta^{n-2}}{1/4}\dd\theta=const.\vspace{-2mm}$$
For $1/2<\rho<1$ we can use\vspace{-2mm}
$$B(r,\rho)\preceq\frac{1}{\rho}\int_{0}^{\pi}\frac{\theta^{n-2}}{\theta^2}\,\mathrm{d}\theta\preceq \frac{1}{\rho}, \mbox{ for } n>3.\vspace{-2mm}$$

This means that, for $n>3$ we have our result. In the case $n=3$ we have that condition $(ii)$ from Lemma \ref{IntLem1} the conclusion.

\end{proof}
\begin{prop}\label{MainProp}
Suppose $n\geq 3$, $\psi\in C(\mathbb{B}^n,\mathbb{R}^n)$ and $|\psi(x)|\leqslant M(1-|x|^2)$ in $\mathbb{B}^n$, where $M$ is a constant. Let $1\leqslant k\leqslant n$ and $u(x)=\int_{\B^n} G_h(x,y)\psi(y)\,\mathrm{d}\tau(y).$ Then\vspace{-2mm}
$$\frac{\partial}{\partial x_k}u(x)=\int_{\B^n} \frac{\partial}{\partial x_k}G_h(x,y)\psi(y)\,\mathrm{d}\tau(y),\vspace{-2mm}$$
and there exists a constant $\beta_1=\beta_1(n,M)$ such that\vspace{-2mm}
$$\left|\frac{\partial}{\partial x_k}u(x)\right|\leqslant\beta_1, \mbox{ for all $x\in\B^n$}.$$
\end{prop}
\bpf
Let us fix $x\in\mathbb{B}_n$ and denote $V^x=u\circ \varphi_x$. Then, by using invariance of $\Delta_h$ we have $\Delta_h V^x (y)=\Delta_h(u\circ \varphi_x)(y)= \Delta_h u(\varphi_x((y)) = \psi(\varphi_x((y)) $ and\vspace{-2mm}
$$V^x(0)=  \int_{\mathbb{B}_n} G_{h}(0,y)  \Delta_h V^x (y)  \dd\tau(y)= \int_{\mathbb{B}_n} G_{h}(0,y) \psi(\varphi_x(y)) \dd\tau(y).\vspace{-2mm}$$
This means that\vspace{-1mm}
$$\frac{\partial}{\partial x_k} V^x(0)= \int_{\mathbb{B}_n}\frac{\partial}{\partial x_k} G_{h}(0,y) \psi(\varphi_x(y)) \dd\tau(y).\vspace{-1mm}$$
Formula (5.3) from \cite{ChenRas} gives us\vspace{-2mm}
$$\frac{\partial}{\partial x_k}G_h(0,y)=\frac{y_k(1-|y|^2)^{n-1}}{n|y|^n}.\vspace{-3mm}$$
Now, we have\vspace{-2mm}
$$\left|\frac{\partial}{\partial x_k} V^x(0)\right|\leqslant\frac 1n\int_{\mathbb{B}_n}\frac{(1-|y|^2)^{n-1}}{|y|^{n-1}}|\psi(\varphi_x(y))| \dd\tau(y)\preceq\int_{\mathbb{B}_n}\frac{(1-|y|^2)^{n-1}}{|y|^{n-1}}(1-  |\varphi_x(y)|^2) \dd\tau(y).\vspace{-2mm}$$
Since  $1- |\varphi_x(y)|^2= \dfrac{(1-|y|^2) (1-|x|^2)}{[x,y]^2}$, we get\vspace{-3mm}
$$\left|\frac{\partial}{\partial x_k} V^x(0)\right|\preceq \int_{\mathbb{B}_n}\frac{\dd \nu(y)}{|y|^{n-1}[x,y]^2}(1-|x|^2)=\int_0^1\dif\rho\int_{\mathbb{S}^{n-1}}\frac{\dd\sigma(\xi)}{1-2\rho\langle x,\xi\rangle+\rho^2|x|^2}(1-|x|^2)\preceq 1-|x|^2,\vspace{-1mm} $$
for $1/2\leqslant |x|< 1$, by Lemma \ref{IntLem2}. Using (cf. \cite{stoll})
\begin{equation}\label{aut-est}
 |\nabla V^x(0)|  \approx |\nabla u(x)| (1-|x|^2),\vspace{-1mm}
\end{equation}
we find that $\nabla u$ is bounded, q.e.d.
%\epf
\epf\vspace{-4mm}
\subsection{Green function and Riesz potential}

At the end of this section, let us investigate some alternative assumptions, which can be stated, instead of condition (h3).\\
Let us introduce the following condition.
\begin{itemize}
    \item[(h3-1)] Function $\psi$ belongs to $C(\mathbb{B}^n)$ and $(1-|y|^2)^{-2} \psi $ belongs  $L^p(\mathbb{B}^n,\nu(y))$ for some $p>n$.
\end{itemize}
Let $\mu\in (0,1)$ and $\Omega$ be a bounded domain in $\mathbb{R}^n$. We define Riesz potential, as the following operator
\vspace{-3mm}
$V_\mu$ on $L_1(\Omega,\dd\nu)$ as \vspace{-2mm}
$$V_\mu(x)=\int_{\Omega}|x-y|^{n(\mu-1)}f(y)\dd \nu(y).\vspace{-2mm}$$
Now, let us state \cite[Lemmma 7.12]{gil.trud}.

\begin{lema}\label{RisLema}
The operator $V_\mu$ maps $L^p(\Omega{,}\dd\nu)$ continuously into $L^q(\Omega{,}\dd\nu)$ for any q, $1{\leqslant} q{\leqslant} +\infty$, satisfying\vspace{-2mm}
\begin{equation}\label{RisCond}
    0\leqslant\delta=\delta(p,q)=\frac 1p-\frac 1q<\mu.\vspace{-2mm}
\end{equation}
\end{lema}

Since
\vspace{-2mm}
$$ | (D_k G_h)_1(x,y)|\preceq  K_0(x,y):=\frac{(1-|x|^2)^{n-1} (1-|y|^2)^{n-1}}{|x-y|^{n-1}[x,y]^n} $$  and
$(1-|x|^2)^{n-1} (1-|y|^2)\leqslant [x,y]^n$, we find
\vspace{-2mm}
  $$ | (D_k G_{h})_1(x,y)|\preceq   K_1(x,y):=\frac{(1{-}|y|^2)^{n{-}2}}{|x{-}y|^{n{-}1}} \mbox{ and } | (D_k G_{h})_2(x,y)|\preceq   K_2(x,y):=\frac{(1{-}|y|^2)^{n{-}3}}{|x{-}y|^{n{-}2}}.\vspace{-2mm}$$
Set\vspace{-2mm}
$$M(x):=\int_{\mathbb{B}_n}  K_1(x,y) |\psi (y)| d\tau(y)=\int_{\mathbb{B}_n}\frac{(1-|y|^2)^{-2}\psi(y)}{|x-y|^{n-1}}\dd\nu(y).$$
Now, we can conclude that $|I_{3,k}(x)|\leqslant M(x)$.\smallskip

If $\psi $ satisfies (h3-1),  we have that the function $M$ is bounded on $\mathbb{B}^n$, by Lemma \ref{RisLema}. In order $|I_{4,k}|$ to be bounded, we need that $(1-|y|^2)^{-3}\psi$ belongs to $L^p(\mathbb{B}^n,\nu(y))$ for some $p>\frac{n}{2}$. It can easily be checked that condition (h3-1) implies boundedness of partial derivatives of the hyperbolic Green potential of the function $\psi$. Thus we have the following statement. Let\vspace{-2mm}
\begin{equation}\label{h4}
  (1-|y|^2)^{-1-\alpha} \psi \in L^p (\mathbb{B}^{n}, d \nu)\mbox{ for some } p>n\mbox{ and } \alpha\in (0,1].\vspace{-2mm}
\end{equation}
\begin{prop}
If   $(\ref{h4})$ holds, then
$G_{h}[\psi]$ is   $\alpha$-H\"{o}lder, $\alpha\in (0,1]$.
\end{prop}
\begin{proof}
Since $(1-|x|^2)^{n-\alpha}(1-|y|^2)^{\alpha}\leqslant[x,y]^n$ we have that
\vspace{-2mm}
$$(1-r)^{1-\alpha} |G_{h}[\psi](r x_0)|\leqslant M_n,\quad 0\leqslant  r <1.\vspace{-2mm}$$
\end{proof}\vspace{-4mm}
\begin{thm}
Suppose that  $n\geqslant 3$ and $u\in
C^{2}(\mathbb{B}^{n},\IR^n) \cap C(\overline{\mathbb{B}^{n}},\IR^n )$,
$u$ satisfies  (h1)    and  $\Delta_h u $  satisfies  $(\ref{h4})$. Then we have the following conclusion:

If  $u$  is  $\alpha$-H\"{o}lder on $\mathbb{S}$, then $u$  is  $\alpha$-H\"{o}lder on $\overline{\mathbb{B}^n}$, for any $0<\alpha\leqslant 1$.
\end{thm}
%%%%%%%%%%%%%%%%%%%%%%%%%%%%%%%%%

\section{Boundary behaviour of the hyperbolic harmonic function on the upper half space}

For $z\in\mathbb{R}^n$ let us write $z=(x,y)$ with $x=(x_1,\ldots,x_{n-1})$. Let us also identified $\mathbb{R}^{n-1}$ with the set $\{z\in\mathbb{R}^n: y=0\}$ and denote upper half-space as $\mathbb{H}^{n}=\{(x,y)\in\mathbb{R}^n:y>0\}$, and $V$ is the $n$-dimensional Lebesgue volume measure on $\mathbb{R}^n$.\\
Now, let us introduce the hyperbolic Poisson's kernel on $\mathbb{H}^{n}$ as
$$P_h(x,y)=\frac{2}{n\omega_n}\left(\frac{y}{|x|^2+y^2}\right)^{n-1}.$$
The Poisson's intgeral formula for the upper half-space asserts that if $f$ is bounded continuous function on $\mathbb{R}^{n-1}$, then the function $u$ defined by \vspace{-1mm}
\begin{equation}\label{HypPoiss1}
    u(x,y)=\frac{2}{n\omega_n}\int_{\mathbb{R}^{n-1}}P_h(x-t,y) f(t)\dd V(t)\vspace{-1mm}
\end{equation}
is hyperbolic harmonic for $y>0$. The constant $\omega_n$ in the formula denotes the volume of the unit ball in $\mathbb{R}^n$. For more details, see \cite[Theorem 5.6.2]{stoll}. In this case, we will shortly write $u=P_h[f]$ in $\mathbb{H}^n$.%Let us observe the following condition \vspace{-1mm}
%\begin{equation}\label{HypInt1}
%    \int_{\mathbb{R}^{n-1}}\frac{|f(t)|}{(1+|t|^2)^{n-1}}\dd t < +\infty, \vspace{-1mm}
%\end{equation}
%where $f\in C(\mathbb{R}^{n-1})$. In our case the boundedness of $f$ can be replaced by the condition (\ref{HypInt1}).\smallskip

Now we will state the lemma, proof is analogous to the proof of {\cite[Lemma 6.1.]{LiTam2}}

\begin{lem}
Let $f$ be a continuous function with compact support in $\mathbb{R}^{n-1}$, and $u$ be its hyperbolic harmonic extension given by (\ref{HypPoiss1}). Then $u$ satisfies
$$
\lim_{y \to 0} y\frac{\partial u}{\partial x_i}(x,y) =0 \quad \mbox{for} \quad 1\leqslant i\leqslant n-1,
\quad\mbox{and}\quad  \lim_{y \to 0} y\frac{\partial u}{\partial y}(x,y) =0,\vspace{-1mm}
$$
and the convergence is uniform in $x\in\mathbb{R}^{n-1}$. If, in addition, $f\in C^1$ then\vspace{-1mm}
$$
\frac{\partial u}{\partial x_i}(x,y)=\frac{2}{n \omega_n} \int_{\mathbb{R}^{n-1}} \left(\frac{y}{|x-t|^2+y^2}\right)^{n-1} \frac{\partial f}{\partial t_i}(t) \, \mathrm{d} V(t)\vspace{-1mm}
$$
for all $1\leqslant i\leqslant n-1$, where $t=(t_1,\ldots,t_{n-1})\in \mathbb{R}^{n-1}$. Hence $\frac{\partial u}{\partial x_i}$ is continuous up to the boundary given by $y=0$ for $1\leqslant i\leqslant n-1$.  Moreover,\vspace{-1mm}
$$
\lim_{y\to 0}\frac{\partial u}{\partial x_i}(x,y)=\frac{\partial f}{\partial x_i}(x).
$$

%If $f$ is differentiable at point $0$, then\vspace{-1mm}
%$$

%$$
\end{lem}

Next, let us state the following condition:\vspace{-3mm}
\begin{itemize}
%    \item[$(i)$] $f$ is bounded function which belongs to $C(\mathbb{R}^{n-1})$,\vspace{-2mm}
    \item[$(i)$] $f$ belongs to $C^1(\mathbb{R}^{n-1})$ and has compact support.\vspace{-2mm}
%    \item[$(iii)$] $\int_{\mathbb{R}^{n-1}}\frac{|f(t)|}{(1+|t|^2)^{n-1}}\dd t < +\infty.$
\end{itemize}

Before we state the main result of this section, let us introduce the following
\begin{lem}\label{UpperHypIntEst}
If we denote\vspace{-2mm} $$I_s^{\alpha}(y)=\int_{\mathbb{R}^{n-1}}\frac{|x|^{\alpha}}{(|x|^2+y^2)^{\frac{s}{2}}}\dd V(x),\vspace{-2mm}$$ we have that $I_s^{\alpha}(y)\preceq \frac{1}{y^{s-n+1-\alpha}}$, for $s>n-1$ and $0<\alpha\leqslant 1$.
\end{lem}
Special case of Lemma \ref{UpperHypIntEst} gives that there exist $M_1(n)>0$ such that\vspace{-2mm} $$y^{n-2}\int_{\mathbb{R}^{n-1}} \frac{|x|}{(|x|^2+y^2)^{n-1}}\dd V(x)\leqslant M_1(n), y>0.\vspace{-2mm}$$
\begin{lema}(\cite{Ru})\label{Sphere}
If $f$ is a continuous and radial function (i.e.
$f(x) = \tilde{f}(|x|))$ on the closed ball of radius $R$ in $\mathbb{R}^n$, centered at the origin, then
$$\int_{B(R)}f(x)\dd V(x)=\sigma_{n-1}\int_0^R \tilde{f}(r)r^{n-1}\dd r.$$

\end{lema}
Now, we can formulate the next Theorem, which will be important for proof of the main result of this section
\begin{thm}\label{thmNorDer}
Let $f$ satisfies condition $(i)$ and, let, in addition, function $f$ be differentiable at the point $0$. If $u$ is defined as (\ref{HypPoiss1}), then \vspace{-2mm}
$$\lim_{y\to 0}\frac{\partial u}{\partial y}(0,y)=0.\vspace{-1mm}$$
\end{thm}
\begin{proof}
After an easy computation, we can get\vspace{-2mm}
$$\frac{\partial}{\partial y}P_h(x,y)=\frac{2(n-1)}{n\omega_n}\frac{y^{n-2}}{(|x|^2+y^2)^{n-1}}\frac{|x|^2-y^2}{|x|^2+y^2}.$$
If we denote $q(x,y):=\dfrac{|x|^2-y^2}{|x|^2+y^2}$, we can see that $|q(z)|\leqslant 1$ on $\mathbb{H}^{n}$. This gives us that\vspace{-2mm}
\begin{equation}\label{HalfPoisNormEst}
\left|\frac{\partial}{\partial y}P_h(x,y)\right|\preceq K_n(z):=\frac{y^{n-2}}{(|x|^2+y^2)^{n-1}}\mbox{ for every }z\in\mathbb{H}^{n}.\vspace{-2mm}
\end{equation}
Let $M_1(n)$ be as in Lemma \ref{UpperHypIntEst}. By the assumption, for every $\epsilon=\epsilon(0)>0$ there exists $\delta=\delta(0)>0$ such that\vspace{-2mm}
$$f(x)-f(0)=a_1x_1+\ldots+a_{n-1}x_{n-1}+\epsilon(x)|x|,\mbox{ where } |\epsilon(x)|<\frac{\epsilon}{2M_1(n)}\frac{n\omega_n}{2}\mbox{ if } |x|<\delta.\vspace{-2mm}$$
Similarly, as in the proof of Theorem \ref{thmloc0}, we have\vspace{-2mm}
\begin{align*}
\frac{\partial}{\partial y}u(0,y)& =\frac{2}{n\omega_n}\int_{\mathbb{R}^{n-1}}\frac{\partial}{\partial y}P_h(x,y)(f(x)-f(0))\dd V(x)=\\
&= \frac{2}{n\omega_n}\int_{B(\delta)}\frac{\partial}{\partial y}P_h(x,y)(a_1x_1+\ldots+a_{n-1}x_{n-1}+\epsilon(x)|x|)\dd V(x)+\\
&+\frac{2}{n\omega_n}\int_{B(\delta)^c}\frac{\partial}{\partial y}P_h(x,y)(f(x)-f(0))\dd V(x).\vspace{-2mm}
\end{align*}
Since $x_k\frac{\partial}{\partial y}P_h(x,y)$ is odd function of the variable $x_k$ and ball $B(\delta)$ is symmetric domain with respect to hyperplane $x_k=0$, for all $1\leqslant k\leqslant n-1$, we have\vspace{-2mm}
$$\frac{\partial}{\partial y}u(0,y)=\frac{2}{n\omega_n}\left(\int_{B(\delta)}\frac{\partial}{\partial y}P_h(x,y)\epsilon(x)|x|\dd V(x)+\int_{B(\delta)^c}\frac{\partial}{\partial y}P_h(x,y)(f(x)-f(0))\d dV(x)\right).\vspace{-2mm}$$
If we define\vspace{-2mm} $$I_{1,n}(y){:=}\int_{B(\delta)}\left|\frac{\partial}{\partial y}P_h(x,y)\right| |\epsilon(x)||x|\dd V(x),\ \  I_{2,n}(y){:=}\int_{B(\delta)^c}\left|\frac{\partial}{\partial y}P_h(x,y)\right| |f(x){-}f(0)|\dd V(x),\vspace{-2mm}$$
we have that $\left|\frac{\partial}{\partial y}u(0,y)\right|\preceq I_{1,n}(y)+I_{2,n}(y)$.
Now, using inequality \ref{HalfPoisNormEst} we get\vspace{-2mm}
\begin{equation}\label{NorEst1}
I_{1,n}(y)\leqslant\epsilon y^{n-2}\int_{\mathbb{R}^{n-1}} \frac{|x|}{(|x|^2+y^2)^{n-1}}\dd V(x)<\frac{\epsilon}{2}\frac{n\omega_n}{2}.\vspace{-2mm}
\end{equation}
Also, we have that $I_{2,n}(y)\preceq y^{n-2}\int_{|x|\geqslant\delta}\frac{\dd V(x)}{(|x|^2+y^2)^{n-1}}.$ Let us denote $J_{\delta,n}(y):=\int_{|x|\geqslant\delta}\frac{\dd V(x)}{(|x|^2+y^2)^{n-1}}.$ After introducing change of variables $x=yt$, we get that\vspace{-3mm}
$$J_{\delta,n}(y){=}\frac{1}{y^{n{-}1}}\int_{|t|\geqslant\frac{\delta}{y}}\frac{\dd V(t)}{(1{+}|t|^2)^{n-1}}{=}\frac{\sigma_{n{-}1}}{y^{n{-}1}}\int\limits_{\delta/y}^{+\infty}\frac{r^{n{-}2}}{(1{+}r^2)^{n-1}}\dd r{=}\left\lceil\rho{=}\frac 1r\right\rfloor{=}\frac{\sigma_{n-1}}{y^{n-1}}\int\limits_{0}^{y/\delta}\frac{\rho^{n-2}}{(1{+}\rho^2)^{n-1}}\dd \rho.$$
Since $1+\rho^2\geqslant 1$ we have that $J_{\delta,n}(y)\preceq \frac{1}{\delta^{n-1}}.$ This means that\vspace{-2mm}
\begin{equation*}
I_{2,n}(y)\leqslant M_2(f,n)\frac{y^{n-2}}{\delta^{n-1}}<\frac{\epsilon}{2}\frac{n\omega_n}{2},\vspace{-2mm}
\end{equation*}
if $0<y<\left(\frac{\epsilon\delta^{n-1}}{2M_2(f,n)}\frac{n\omega_n}{2}\right)^{1/(n-2)}=:\delta_0=\delta_0(\epsilon(0),\delta(0))$. Finally, we have\vspace{-3mm}
\begin{equation}\label{NorEst2}
\left|\frac{\partial}{\partial y}u(0,y)\right|<\epsilon,\mbox{ for } 0<y<\delta_0.\vspace{-2mm}
\end{equation}
\end{proof}
\begin{lem}\label{NorDerCont}
Let $f$ and $u$ be functions defined as in Theorem \ref{thmNorDer} and let $f$ be differentiable in the neighbourhood of the point $0$ and let every partial derivative of function $f$ be continuous at point $0$. Then, for every $\epsilon>0$ there exist $\eta,\delta_0>0$ such that,  \vspace{-2mm}
$$\left|\frac{\partial}{\partial y} u(x,y)\right|<\epsilon,\mbox{ for every } |x|<\eta, 0<y<\delta_0.\vspace{-2mm}$$
In fact, $\lim\limits_{\substack{z\to 0,\\ z\in\mathbb{H}^n}}\frac{\partial}{\partial y} u(z)=0.$
\end{lem}
\begin{proof}
First, for every $\epsilon>0$ there exist $\delta,\eta>0$, such that\vspace{-2mm}
$$f(x+h)-f(x)=\frac{\partial}{\partial x_1}f(x)h_1+\ldots+\frac{\partial }{\partial x_{n-1}}f(x)h_{n-1}+\epsilon(x,h)|h|,\mbox{ and } |\epsilon(x,h)|{<}\epsilon\mbox{ if } |x|{<}\eta,|h|{<}\delta.\vspace{-2mm}$$
This can easily be proved, using the Mean Value Theorem \cite[Corollary 10.2.9]{ttaoA2}. Namely, we use that, there are $\eta>0$ and $\delta>0$(which does not depend of $s,t$) such that\vspace{-2mm}
$$|f(t+h_ie_i)-f(t)-\frac{\partial }{\partial x_i}f(0)h_i|<\frac{\epsilon}{2(n-1)}|h_i|\mbox{ for } |t|<\eta,|h_i|<\delta,1\leqslant i\leqslant n-1. \vspace{-2mm}$$
Note that, in this case, we need that $\left|\frac{\partial }{\partial x_i}f(s)-\frac{\partial }{\partial x_i}f(0)\right|<\frac{\epsilon}{2(n-1)}$, for all $|s|<\eta+\delta$. Here, $e_1,e_2,\ldots e_{n-1}$ is the standard orthonormal base of $\mathbb{R}^{n-1}$. Now, we are left only to use the fact that\vspace{-2mm}
 \begin{align*}
     f(x{+}h){-}f(x)=&(f(x{+}h){-}f(x{+}h{-}h_1e_1))+\\
     &+(f(x{+}h{-}h_1e_1){-}f(x{+}h{-}h_1e_1{-}h_2e_2)){+}\ldots{+}\\
     &+(f(x{+}h_{n-1}e_{n-1}{-}f(x)).\vspace{-2mm}
 \end{align*}
 %where $e_1,e_2,\ldots e_{n-1}$ is standard orthonormal base of $\mathbb{R}^{n-1}$.
 and,\vspace{-2mm}
 \begin{align*}
     &\frac{\left|f(x{+}h){-}f(x){-}\frac{\partial}{\partial x_1}f(x)h_1{-}\ldots{-}\frac{\partial }{\partial x_{n-1}}f(x)h_{n-1}\right|}{|h|}\leqslant\\
     \frac{1}{|h|}&\left\{ \left|f(x{+}h){-}f(x{+}h{-}{h_1e_1})-{\frac{\partial }{\partial x_1}}f(0)h_1\right|{+}\left|{\frac{\partial }{\partial x_1}}f(0){-}\frac{\partial }{\partial x_1}f(x)\right||h_1|\right.+\\
     &{+}\left|f(x{+}h{-}h_1e_1){-}f(x{+}h{-}{h_1e_1}{-}{h_2e_2})-{\frac{\partial }{\partial x_2}f(0)}h_2\right|+\\
     &+\left|\frac{\partial }{\partial x_2}f(0){-}\frac{\partial }{\partial x_2}f(x)\right||h_2|+\ldots+\\
     &+\left|f(x{+}h_{n-1}e_{n-1}){-}f(x){-}\frac{\partial}{\partial x_{n-1}}f(0)h_{n-1}\right|{+}\\
     &\left.+\left|{\frac{\partial}{\partial x_{n-1}}}f(0){-}{\frac{\partial}{\partial x_{n-1}}}f(x)\right||h_{n-1}|\right\}\leqslant\\
     &\leqslant\frac{\epsilon(|h_1|+\ldots+|h_{n-1}|)}{(n-1)|h|}<\epsilon, \mbox{ for } |x|<\eta, |h|<\delta.
 \end{align*}
Now, we use that in Equation \ref{NorEst2}, $\delta_0$ depends of $\delta$, which has been chosen from the definition of diffenetiability of function $f$ at $0$. Hence, we found an $\eta$-neighbourhood of point $0$ in which $\delta(x)$, from definition of differentiability of $f$ at $x$, actually, does not depend of $x$, we can use the proof of Theorem \ref{thmNorDer} to get our result.
\end{proof}
From the Lemma \ref{NorDerCont} we can conclude that
\begin{thm}
Let $f\in C_c^1(\mathbb{R}^{n-1})$ and $u=P_h[f]$. Then $u\in C^1(\overline{\mathbb{H}^{n}}).$
\end{thm}
\section{Appendix}

\subsection{General kernels $P_{\alpha,\beta}$ of Poisson's type}
In \cite{leut} author introduced special Riemannian metrics on unit ball and upper half-space in $\mathbb{R}^n$ and corresponding Laplace-Beltrami operators. Following methods described in this article, one can investigate boundary behaviour of solutions to corresponding Dirichlet's problems. Namely, Poisson's kernel for this kind of operators, may be written in the form\vspace{-3mm}
\begin{equation*}
P_{\alpha,\beta}(x,y)=\frac{(1-|x|^2)^\alpha}{|x-y|^{2\beta}},\vspace{-2mm}
\end{equation*}
for $\alpha,\beta \in \mathbb{R}$,  $\beta>0$.
If we wish to emphasize that $x$ is a fixed for a moment, sometime we write  $x_0$  instead of $x$.
Set $P=P_{\alpha,\beta}$, $d=d(x)= 1-|x|^2$,  $E=E(x,y)= |x-y|^2$,  $p=d_\alpha=d^{\alpha}$,  $q=E_\beta= E^\beta$ and $X= X(x_0,v^0,y)= \langle \nabla_x P_{\alpha,\beta}(x_0,y),v^0\rangle$,
where $x_0\in\mathbb{ B}_p$  and  $v^0\in T_{x_0} \mathbb{B}_p$ is the unit vector.

Check that  $X (x,y)=2 P Y$, where\vspace{-2mm}
$$Y=Y(x,y)=\beta \frac{\langle y,v^0 \rangle  -a_0}{E} - \alpha \frac{a_0}{d}=\beta \frac{\langle y,v^0 \rangle}{E} - b_0 ,\vspace{-2mm}$$
$a_0=a_0(x_0,v^0)= \langle x_0,v^0 \rangle$  and $b_0=\left(\frac{\beta}{E} +  \frac{\alpha}{d}\right)a_0$.   In particular,  $a_0(0,v^0)= 0$,  $b_0(0,v^0)= 0$.\smallskip

%%%%%%%%%%%%%%%%%%%%%%%%%%%%%%%%%%%%%%%%%%
In harmonic case, $\alpha=1$, $\beta=n/2$    and     for $y\in \mathbb{S}$, $P(0,y)=1$, $Y(0,y,l)=\frac{n}{2} \langle y,l\rangle$  and therefore
if $v$ is harmonic bounded on $\mathbb{B}$,
$D_lv(0)= \langle \nabla_x v (0),l  \rangle= n \int_{S}\langle y,l\rangle v^* d\sigma (y)$.

In hyperbolic harmonic case, $\alpha=\beta=n-1$.
\medskip

In \cite{Ge} Geller
introduced a family of diferential operators,\vspace{-2mm}
$$ \Delta_{\alpha,\beta}=(1-|z|^2) {\sum _{i,j}(\delta_{ij} - z_i \overline{z}_j D_i \overline{D}_j + \alpha R + \beta \overline{R} -\alpha  \beta)},\vspace{-4mm}$$
where  $R=\sum z_i D_i$.   $\varphi_z$ is the automorphism of
the ball that maps $z$ to $0$, and such that  $\varphi_z^2
= \mathrm{Id}$ (see \cite[pg 297.]{Ru})
\be
d\lambda (z)= \frac{1}{(1-|z|^2)^{n+1}} \dd V(z)
\ee
 For  Dirichlet problem see formula (3.5) in \cite{Ahern}:
 \be
 u(z)= \int_{S^n}u(\zeta)P_{\alpha,\beta}(z,\zeta)\dd A(\zeta)+   \int_{\mathbb{B}^n}\Delta_{\alpha,\beta} u(\omega)(1-\overline{z}\omega)^\alpha(1-z\overline{\omega})^\beta  (1-|\omega|^2)^{-\alpha-\beta} \dd\lambda (\omega)
 \ee
For instance, it holds if
$u\in C^2(\mathbb{B}^n) \cap C(\overline{\mathbb{B}^n})$ and
 \be
 \int_{\mathbb{B}^n}|\Delta_{\alpha,\beta} u(\omega)|\frac{\dd V(\omega)}{1-|\omega|}< +\infty.
  \ee

\textbf{Acknowledgments} The authors are indebted to M. Arsenovi\'c for an interesting discussions on this paper.


\begin{thebibliography}{1}

\bibitem{Ahern}
	{\sc P. {Ahern}, J. {Bruna}, C. {Cascante}},
	\newblock {\(H^ p\)-theory for generalized \(M\)-harmonic functions in the unit
		ball, Indiana University Mathematics Journal, 45 (1), 103–135 (1996).}
		
\bibitem{ABR1992}
S. Axler, P. Bourdon and W. Ramey, \emph{Harmonic Function Theory},(second edition) Springer Verlag, New York, 2020.
		
\bibitem{ChenRas} {\sc J. Chen, M. Huang, A. Rasila,
X. Wang}, {\it On Lipschitz continuity of solutions of hyperbolic Poisson’s equation},  Calc. Var., 57:13 (2018),  https://doi.org/10.1007/s00526-017-1290-x.
%Lipschitz_cont_sol_hyperbolic.pdf

\bibitem [GT]{gil.trud}  {\sc  D.  Gilbarg, N. Trudinger},
{\it   Elliptic partial Differential  Equation of Second Order}, Second Edition, (1983).

%\bibitem{GruterWidman} {\sc M. Grüter, K. Widman}, {\it The Green function for uniformly elliptic equations}, Manuscripta Math 37, 303–342 (1982).

\bibitem[MSS]{mss} {\sc M. Mateljevi\'c, R. Salimov, E. Sevostyanov},   {\it  H\"{o}lder and Lipschitz  continuity in  Orlicz-Sobolev  classes, the distortion and  harmonic   mappings}, to appear in Filomat (accepted 2021).

\bibitem{LiTam2} {\sc P. Li, L. Tam}, {\it Uniqueness and Regularity of Proper Harmonic Maps II}, Indiana University Mathematics Journal, 42 (2), 591–635 (1993).

\bibitem{ahlMob}{\sc L. V. Ahlfors}, {Mobius Transformations in Several Dimensions}, University of Minnesota (1989).

\bibitem{stoll} {\sc M. Stoll}, {\it Harmonic and Subharmonic Function Theory on the Hyperbolic Ball} (London Mathematical Society Lecture Note Series), Cambridge: Cambridge University Press (2016).

%\bibitem{Sphere} {\sc W. Rudin} {\it Function theory in the unit ball of $\mathbb{C}^n$}, (Grundlehren der mathematischen Wissenschaften; 241)

\bibitem{ttaoA2} {\sc T. Tao} {\it Analysis I}, Texts and Readings in Mathematics, Volume 37, ISBN:978-93-80250-64-9 © Hindustan Book Agency 2015 © Springer Science+Business Media Singapore (2016).

\bibitem{HypHar2} {\sc J. Chen, M. Huang, S. Lee, X. Wang,} {\it Equivalent Norms of Solutions to Hyperbolic Poisson’s Equations},. J Geom Anal 31, 8173–8201 (2021). https://doi.org/10.1007/s12220-020-00581-1

\bibitem{Mat} {\sc M. Mateljevi\'{c}}, {\it Boundary Behaviour of Partial Derivatives for Solutions to Certain Laplacian-Gradient Inequalities and Spatial QC Maps},  Operator Theory and Harmonic Analysis, Springer Proceedings in Mathematics and Statistics, 357, 393--418 (2021).

\bibitem{leut} {\sc H. Leutwiler}, {\it Best constants in the Harnack inequality for the Weinstein equation} Aeq. Math. 34, 304–315 (1987). https://doi.org/10.1007/BF01830680

\bibitem{Ge} {\sc D. Geller}, {\it Some results in H theory for the Heisenberg group}, Duke Math. J. 47 (1980), 365--391

\bibitem{Ru} {\sc W. Rudin}, {\it Function Theory in the Unit Ball}, Springer-Verlag, New York, 1980.

\bibitem{Pri} {\sc D. Li, L. Wang}, {\it Elliptic equations on convex domains with nonhomogeneous Dirichlet boundary conditions.} Journal of Differential Equations 246 (2009): 1723-1743.

\bibitem{MMNMBound} {\sc M. Mateljevic, N. Mutavdžic}, {\it The Boundary Schwarz lemma for harmonic and pluriharmonic mappings and some generalisations.}  arXiv:2111.02384v5 [math.CV]

\bibitem{JChen} {\sc P. Li, J. Chen, X. Wang}, {\sc Quasiconformal solutions of Poisson equations.} Bull. Aust. Math. Soc. 92, 420–428 (2015)

\bibitem{kal} {\sc D. Kalaj, M. Mateljevi\'{c}}, {\it Inner estimate and quasiconformal harmonic maps between smooth domains.}
J. Anal. Math. 100, 117–132 (2006)

\bibitem{shen-yau} {\sc R. Schoen, S. T. Yau}, {\it Lectures on Harmonic Maps}. Cambridge, MA: International Press; 1997

\bibitem{MAVMMM} {\sc M. Arsenovi\'{c}, V. Manojlovi\'{c}, M. Mateljevi\'{c}}, {\it Lipschitz-type spaces and harmonic mappings in the space.} Ann. Acad. Sci. Fenn., Math. 35, 379–387 (2010)
\end{thebibliography}
\end{document}